\documentclass[10pt]{article}
\usepackage{amsmath,amsthm,amsfonts,latexsym,amsopn,verbatim,amscd,amssymb}
\theoremstyle{plain}
\newtheorem{theorem}{Theorem}[section]
\newtheorem{lemma}[theorem]{Lemma}
\newtheorem{corollary}[theorem]{Corollary}

\newtheorem{proposition}[theorem]{Proposition}

\newcommand{\bnum}{\begin{enumerate}}
\newcommand{\enum}{\end{enumerate}}

\numberwithin{equation}{section}

\begin{document}

\title{On relative commuting probability of  finite rings}
\author{Parama Dutta and Rajat Kanti Nath\footnote{Corresponding author}}
\date{}
\maketitle
\begin{center}\small{\it
Department of Mathematical Sciences, Tezpur University,\\ Napaam-784028, Sonitpur, Assam, India.\\

Emails:\, parama@gonitsora.com and rajatkantinath@yahoo.com}
\end{center}

\medskip

\begin{abstract}
In this paper we study the probability that the commutator of a randomly chosen pair of elements, one from a subring of a finite ring and other from the ring itself equals to a given element of the ring.
\end{abstract}

\medskip

\noindent {\small{\textit{Key words:}  finite ring, commuting probability, ${\mathbb{Z}}$-isoclinism of rings.}}

\noindent {\small{\textit{2010 Mathematics Subject Classification:}
16U70, 16U80.}}

\medskip

\section{Introduction}
Let $S$ be a subring of a finite ring $R$. The relative commuting probability of $S$ in $R$ denoted by ${\Pr}(S, R)$ is the probability that a randomly chosen pair of elements one from $S$ and the other from $R$ commute. That is 
\[
\Pr(S, R) = \frac{|\lbrace(x, y)\in S \times R : [x, y]  = 0 \rbrace|}{|S||R|}
\]
where $0$ is the additive identity of $R$ and $[x, y] = xy - yx$ is the commutator of $x$ and $y$. The study of ${\Pr}(S, R)$ was initiated in \cite{jutireka}. Note that ${\Pr}(R, R) = {\Pr}(R)$, a notion called  commuting probability of $R$ introduced by Machale \cite{dmachale}  in the year 1976.  
It may be mentioned here that the commuting probability of algebraic structures was originated from the works of Erd$\ddot{\rm o}$s and Tur$\acute {\rm a}$n  \cite{pEpT68} in the year 1968.
In this paper we consider the probability that the commutator of a randomly chosen pair of elements, one from the subring $S$ and the other from $R$, equals a given element $r$ of $R$. We write ${\Pr}_r(S, R)$ to denote this probability. Therefore
\begin{equation}\label{mainformula}
{\Pr}_r(S, R) = \frac{|\lbrace(x, y)\in S\times R : [x,y] = r\rbrace|} {|S||R|}.
\end{equation}
Clearly ${\Pr}_r(S, R) = 0$  if and only if $r \notin K(S,R) := \{[x,y]:x\in S,y\in R\}$. Therefore we consider $r$ to be an element of $K(S,R)$ throughout the paper. Also ${\Pr}_0(S, R) = {\Pr}(S, R)$ and ${\Pr}_r(R) := {\Pr}_r(R, R)$, a notion studied in \cite{dn2}. 
In this paper we obtain some computing formulas and bounds for ${\Pr}_r(S,R)$. We also discus an invariance property of ${\Pr}_r(S,R)$ under $\mathbb{Z}$-isoclinism.
The motivation of this paper lies in \cite{nY2015} where analogous generalization of commuting probability of finite group is studied.

We write $[S, R]$ and $[x,R]$ for $x \in S$  to denote the additive subgroups of $(R, +)$ generated by the sets $K(S,R)$ and $\lbrace [x, y] : y\in R\rbrace$ respectively. Note that
$[x,R] = \{[x, y] : y\in R\}$. Let  $Z(S, R) := \{x \in S : xy = yx\,  \forall y \in R\}$.
Then $Z(R) := Z(R,R)$ is the center of $R$. Further, if $r \in R$ then the set $C_S(r) := \{x\in S : xr = rx\}$ is a subring of $S$ and $\underset{r \in R}{\cap} C_S(r) = Z(S,R)$. We write $\frac{R}{S}$ and $|R : S|$ to denote the additive quotient group  and the index of $S$ in $R$.

\section{Computing formula for ${\Pr}_r(S, R)$}
In this section, we derive some computing formulas for ${\Pr}_r(S, R)$. We begin with the following useful lemmas.

\begin{lemma}{\rm \cite[Lemma 2.1]{dn2}\label{lemma1}}
Let $R$ be a  a finite ring. Then 
\[
|[x, R]| = |R : C_R(x)| \text{ for all } x \in R.
\]
\end{lemma}

\begin{lemma}\label{lemma2}
Let $S$ be a subring of a finite ring $R$ and $T_{x, r}(S, R) = \{y\in R : [x, y] = r\}$ for $x\in S$ and $r\in R$. Then we have the followings 
\begin{enumerate}
\item  $T_{x, r}(S, R) \ne \phi$ if and only if $r \in [x, R]$.
\item If $T_{x, r}(S, R)\neq \phi$ then $T_{x, r}(S, R) = t + C_R(x)$ for some $t\in T_{x, r}$.
\end{enumerate}
\end{lemma}

\begin{proof}
Part (a) follows from the fact that $y \in T_{s, r}(S, R)$ if and only if $r \in [s, R]$. Let $t \in T_{x, r}(S, R)$ and  $p \in t + C_R(x)$. Then $[x, p] = r$ and so $p \in T_{x, r}(S, R)$. Therefore,  $t + C_R(x) \subseteq T_{x, r}(S, R)$. Again, if $y \in T_{x, r}(S, R)$ then  $(y - t) \in C_R(x)$ and so $y \in t + C_R(x)$. Therefore, $t + C_R(x)\subseteq T_{x, r}(S, R)$. Hence part (b) follows.
\end{proof}
\noindent Now we state and prove the following  main results of this section.

\begin{theorem}\label{com-thm}
Let $S$  be a subring of  a finite ring $R$. Then
\[
{\Pr}_r(S,R) = \frac {1}{|S||R|}\underset{r\in [x, R]}{\underset{x\in S}{\sum}}|C_R(x)| = \frac {1}{|S|}\underset{r\in [x, R]}{\underset{x\in S}{\sum}}\frac{1}{|[x, R]|}.
\]
\end{theorem}

\begin{proof}
Note that $\{(x, y) \in S\times R : [x, y] = r\} = \underset{x\in S}{\cup}(\{x\}\times T_{x, r}(S, R))$. Therefore, by \eqref{mainformula} and Lemma \ref{lemma2}, we have

\begin{equation} \label{comfor1}
|S||R|{\Pr}_r(S,R) = \underset{x \in S}{\sum} |T_{x, r}(S, R)| = \underset{r\in [x, R]}{\underset{x\in S}{\sum}}|C_R(x)|.
\end{equation}
The second part follows  from \eqref{comfor1} and Lemma  \ref{lemma1}.
\end{proof}

\begin{proposition}\label{symmetricity}
Let $S$ be a subring of a finite ring $R$ and $r \in  R$. Then
${\Pr}_r(S, R) = {\Pr}_{-r}(R, S)$. However, if $2r = 0$  then ${\Pr}_r(S, R) = {\Pr}_{r}(R, S)$.
\end{proposition}

\begin{proof}
Let $X = \{(x, y) \in S \times R : [x, y] = r\}$ and $Y = \{(y, x) \in R \times S : [y, x] = -r\}$. It is easy to see that $(x, y) \mapsto (y, x)$ defines a bijective mapping from $X$ to $Y$. Therefore, $|X| = |Y|$ and the result follows from \eqref{mainformula}.

Second part follows from the fact that $r = -r$ if $2r = 0$.
\end{proof}

\begin{proposition}
Let $S_1$ and $S_2$ be two  subrings of the finite rings $R_1$ and $R_2$ respectively. If $(r_1, r_2) \in R_1 \times R_2$  then
\[
{\Pr}_{(r_1, r_2)}(S_1 \times S_2, R_1\times R_2) = {\Pr}_{r_1}(S_1, R_1){\Pr}_{r_2}(S_2, R_2).
\]
\end{proposition}

\begin{proof}
Let $X_i = \{(x_i, y_i) \in S_i\times R_i : [x_i, y_i] = r_i\}$ for $i = 1, 2$ and $Y = \{((x_1, x_2), (y_1, y_2)) \in (S_1\times S_2) \times (R_1\times R_2) : [(x_1, x_2),(y_1, y_2)]= (r_1, r_2)\}$. Then $((x_1, y_1), (x_2, y_2)) \mapsto ((x_1, x_2),(y_1, y_2))$ defines a bijective map from $X_1 \times X_2$ to $Y$. Therefore, $|Y| = |X_1||X_2|$ and hence the result follows from \eqref{mainformula}.
\end{proof}

\noindent Using Proposition \ref{symmetricity} in Theorem \ref{com-thm}, we get the following corollary.
\begin{corollary}\label{formula1}
Let $S$  be a subring of a finite ring  $R$. Then
\[
{\Pr}(R, S) = {\Pr}(S, R) = \frac {1}{|S||R|}\sum_{x\in S}|C_R(x)|  = \frac {1}{|S|}\sum_{x\in S}\frac{1}{|[x, R]|}.
\]
\end{corollary}

We conclude this section with the following corollary.
\begin{corollary}
Let $S$ be a subring of a finite non-commutative ring $R$.  If $|[S, R]|= p$, a prime, then
\[
{\Pr}_r(S, R) = \begin{cases}
\frac{1}{p}\left(1 + \frac{p - 1}{|S:Z(S,R)|}\right), &      
                                          \mbox{if } r = 0 \\
\frac{1}{p}\left(1 - \frac{1}{|S:Z(S,R)|}\right), &\mbox{if }           
                                       r \neq 0.
\end{cases}
\]
\end{corollary}

\begin{proof}
For $x\in S \setminus Z(S,R)$, we have $\{0\} \subsetneq [x,R] \subseteq [S,R]$. Since $|[S, R]|= p$, it follows that $[S, R] = [x, R]$ and hence $|[x, R]|= p$ for all $x\in S \setminus Z(S,R)$.

If $r = 0$ then by Corollary \ref{formula1}, we have
\begin{align*}
{\Pr}_r(S,R) = & \frac{1}{|S|}\left(|Z(S,R)| + \sum_{x\in S\setminus Z(S,R)}\frac{1}{|[x, R]|}\right)\\ 
= & \frac{1}{|S|}\left(|Z(S,R)| + \frac{1}{p}(|S| - |Z(S,R)|)\right) \\
    = & \frac{1}{p}\left(1 + \frac{p - 1}{|S:Z(S,R)|}\right).
\end{align*}

If $r\neq 0$ then $r\notin [x,R]$ for all $x \in Z(S,R)$ and $r\in [x,R]$ for all $x\in S \setminus Z(S,R)$.
Therefore, by Theorem \ref{com-thm}, we have

\begin{align*}
{\Pr}_r(S,R) = &  \frac {1}{|S|}\underset{x\in S \setminus Z(S,R)}{\sum}\frac{1}{|[x, R]|}
=  \frac{1}{|S|}\underset{x\in S \setminus Z(S,R)}{\sum}\frac{1}{p}\\
= &  \frac{1}{p}\left(1 - \frac{1}{|S:Z(S,R)|}\right).
\end{align*}
Hence, the result follows.
\end{proof}

\section{Bounds for ${\Pr}_r(S, R)$}
If $S$ is a subring of a finite ring $R$ then it was shown in  \cite[Theorem 2.16]{jutireka} that 
\begin{equation}\label{lbThm2.5}
{\Pr}(S, R) \geq \frac {1}{|K(S,R)|}\left(1 + \frac{|K(S,R)| - 1}{|S : Z(S,R)|}\right).
\end{equation}
Also, if $p$ is the smallest prime dividing $|R|$ then by   \cite[Theorem 2.5]{jutireka} and  \cite[Corollary 2.6]{jutireka} we have  
\begin{equation}\label{ubCor2.6}
{\Pr}(S, R) \leq \frac{(p -1)|Z(S,R)| + |S|}{p|S|} 
\text{ and }{\Pr}(R) \leq \frac{(p -1)|Z(R)| + |R|}{p|R|}.
\end{equation}
In this section, we obtain several bounds for ${\Pr}_r(S, R)$ and show that some of our bounds are better than the bounds given in \eqref{lbThm2.5} and \eqref{ubCor2.6}. We begin with the following upper bounds.
\begin{proposition}
Let $S$  be a subring of a finite ring $R$. If $p$ is the smallest prime dividing $|R|$ and $r \ne 0$ then
\[
{\Pr}_r(S, R)\leq \frac {|S| - |Z(S, R)|}{p|S|} < \frac {1}{p}.
\]
\end{proposition}
\begin{proof}
Since  $r \ne 0$ we have $S \ne Z(S, R)$. If $x \in Z(S, R)$ then $r \notin [s, R]$. If $x \in S \setminus Z(S, R)$ then $C_R(x) \ne R$. Therefore, by Lemma \ref{lemma1}, we have  $|[x, R]| = |R : C_R(x)| > 1$. Since   $p$ is the smallest prime dividing $|R|$ we have $|[x, R]| \geq p$.  Hence the result follows from Theorem \ref{com-thm}.
\end{proof}

\begin{proposition}\label{ub02}
Let $S$ be a subring of a finite ring $R$. Then ${\Pr}_r(S, R)$ $\leq {\Pr}(S, R)$ with equality if and only if $r = 0$. 
\end{proposition}
\begin{proof}
By Theorem \ref{com-thm} and Corollary \ref{formula1}, we have
\[
{\Pr}_r(S, R) = \frac {1}{|S||R|}\underset{r \in [x, R]}{\underset{x\in S}{\sum}}|C_R(x)|
 \leq \frac {1}{|S||R|}\underset{x\in S}{\sum}|C_R(x)|
=  \Pr(S, R).
\]
The equality holds if and only if $r = 0$.
\end{proof}

\begin{proposition}\label{ub03}
If $S_1\subseteq S_2$ are two subrings of a finite ring $R$. Then
\[
{\Pr}_r(S_1, R) \leq |S_2 : S_1|{\Pr}_r(S_2, R).
\]
\end{proposition}
\begin{proof}
By Theorem \ref{com-thm}, we have
\begin{align*}
|S_1||R|{\Pr}_r(S_1, R) = &\underset{r  \in [x, R]}{\underset{x \in S_1}{\sum}}|C_{R}(x)|\\
\leq &\underset{r \in [x, R]}{\underset{x\in S_2}{\sum}}|C_{R}(x)| = |S_2||R|{\Pr}_r(S_2,R).
\end{align*}
Hence the result follows.
\end{proof}
\noindent Note that  equality holds in Proposition \ref{ub03} if and only if
$r \notin  [x, R]$   for all   $x\in S_2 \setminus S_1$. If $r = 0$ then the condition of equality reduces to $S_1 = S_2$. Putting $S_1 = S$ and $S_2 = R$ in Proposition \ref{ub03} we have the following corollary.

\begin{corollary}
If $S$ is a  subring of a finite ring $R$ then
\[
{\Pr}_r(S, R) \leq |R : S|{\Pr}_r(R).
\]
\end{corollary}

For any subring $S$ of $R$, let $m_S = \min\{|[x,R]|: x\in S\setminus Z(S,R)\}$ and $M_S = \max\{|[x,R]|: x\in S\setminus Z(S,R)\}$. In the following theorem we give  bounds for ${\Pr}(S,R)$ in terms of  $m_S$ and $M_S$.

\begin{theorem}\label{compare}
Let $S$ be a subring of a finite ring  $R$. Then
  \[
  \frac{1}{M_S}\left(1 + \frac{M_S - 1}{|S : Z(S,R)|}\right)\leq \Pr(S,R) \leq  \frac{1}{m_S}\left(1 + \frac{m_S - 1}{|S : Z(S,R)|}\right).
  \]
  The equality holds if and only if $m_S = M_S = |[x,R]|$ for all $x\in S\setminus Z(S,R)$.
\end{theorem}

\begin{proof}
Since $m_S \leq |[x,R]|$ and $M_S \geq |[x,R]|$ for all $x \in S\setminus Z(S,R)$,  we have
\begin{equation}\label{eqm_s-1}
\frac{|S| - |Z(S, R)|}{M_S} \leq \sum_{x\in S \setminus Z(S,R)}\frac{1}{|[x, R]|} \leq \frac{|S| - |Z(S, R)|}{m_S}.
\end{equation}
Again, by Corollary \ref{formula1}, we have

\begin{equation}\label{eqm_s-2}
\Pr(S,R) = \frac {1}{|S|}\left(|Z(S,R)| + \sum_{x\in S \setminus Z(S,R)}\frac{1}{|[x, R]|}\right).
\end{equation}
Hence, the result follows from \eqref{eqm_s-1} and \eqref{eqm_s-2}.
\end{proof}

\noindent Note that for any two integers  $m \geq n$, we have
\begin{equation}\label{lemma4.4}
\frac {1}{n}\left(1 + \frac{n - 1}{|S : Z(S,R)|}\right) \geq \frac {1}{m}\left(1 + \frac{m - 1}{|S : Z(S,R)|}\right).
\end{equation}
Clearly equality holds in \eqref{lemma4.4} if $Z(S,R) = S$. Further, if $Z(S,R)\neq S$ then equality holds  if and only if $m=n$. Since $|K(S,R)|\geq M_S$, by \ref{lemma4.4}, it follows that
\[
\frac {1}{M_S}\left(1 + \frac{M_S - 1}{|S : Z(S,R)|}\right) \geq \frac {1}{|K(S,R)|}\left(1 + \frac{|K(S,R)| - 1}{|S : Z(S,R)|}\right).
\]
Therefore, the lower bound obtained in Theorem \ref{compare} is better than the lower bound given in \eqref{lbThm2.5} for $\Pr(S, R)$. Again, if $p$ is the smallest prime divisor of $|R|$ then  $p \leq m_S$ and hence, by \eqref{lemma4.4}, we have
\[
\frac{1}{m_S}\left(1 + \frac{m_S - 1}{|S : Z(S,R)|}\right)\leq \frac{(p -1)|Z(S,R)| + |S|}{p|S|}.
\]
This shows that the upper bound obtained in Theorem \ref{compare} is better than the upper bound given in \eqref{ubCor2.6} for $\Pr(S, R)$.

Putting $S = R$ in Theorem \ref{compare} we have the following corollary.

\begin{corollary}\label{compare1}
Let $R$ be a finite  ring. Then
  \[
  \frac{1}{M_R}\left(1 + \frac{M_R - 1}{|R : Z(R)|}\right)\leq \Pr(R) \leq  \frac{1}{m_R}\left(1 + \frac{m_S - 1}{|R : Z(R)|}\right).
  \]
  The equality holds if and only if $m_R = M_R = |[x,R]|$ for all $x\in R\setminus Z(R)$.
\end{corollary}

We conclude this section noting that  the lower bound obtained in Corollary \ref{compare1} is better than the lower bound obtained in \cite[Corollary 2.18]{jutireka}. Also, if $p$ is the smallest prime divisor of $|R|$ then the upper bound obtained in Corollary \ref{compare1} is better than the upper bound given in \eqref{ubCor2.6} for $\Pr(R)$.

\section{$\mathbb{Z}$-isoclinism and ${\Pr}_r(S, R)$}

The idea of isoclinism of groups was introduced by  Hall \cite{pH40} in 1940. Years after in 2013, Buckley et al. \cite{BMS} introduced $\mathbb{Z}$-isoclinism of rings. Recently, Dutta et al. \cite{jutireka} have introduced $\mathbb Z$-isoclinism  between two pairs of rings, generalizing the notion of  $\mathbb Z$-isoclinism of rings. Let $S_1$ and $S_2$ be two subrings of the rings $R_1$ and $R_2$ respectively.  Recall that a pair of mappings $(\alpha,\beta)$ is called a $\mathbb Z$-isoclinism between $(S_1, R_1)$ and $(S_2, R_2)$ if $\alpha :\frac {R_1}{Z(S_1, R_1)}\rightarrow \frac {R_2}{Z(S_2, R_2)}$ and $\beta :[S_1, R_1]\rightarrow [S_2, R_2]$ are additive group isomorphisms such that $\alpha \left(\frac {S_1}{Z(S_1, R_1)}\right) = \frac {S_2}{Z(S_2, R_2)}$  and $\beta ([x_1, y_1])=[x_2, y_2]$ whenever $x_i \in S_i$,  $y_i\in R_i$ for  $i = 1, 2$; $\alpha (x_1 +  Z(S_1, R_1)) = x_2 +  Z(S_2, R_2)$ and $\alpha (y_1 +  Z(S_1, R_1)) = y_2 +  Z(S_2, R_2)$. Two pairs of rings are said to be $\mathbb Z$-isoclinic if there exists a $\mathbb Z$-isoclinism between them.

In \cite[Theorem 3.3]{jutireka}, Dutta et al.  proved that $\Pr(S_1, R_1) = \Pr(S_2, R_2)$ if the rings $R_1$ and $R_2$ are finite and  the pairs $(S_1, R_1)$ and $(S_2, R_2)$ are $\mathbb Z$-isoclinic. We conclude this paper with the following  generalization of \cite[Theorem 3.3]{jutireka}.

\begin{theorem}
Let $S_1$ and $S_2$  be two subrings of the finite  rings $R_1$ and $R_2$ respectively. If  $(\alpha,\beta)$ is a    $\mathbb Z$-isoclinism between $(S_1, R_1)$ and $(S_2, R_2)$ then
\[
{\Pr}_r(S_1, R_1) = {\Pr}_{\beta (r)}(S_2, R_2).
\]
\end{theorem}
\begin{proof}
By Theorem \ref{com-thm}, we have
\begin{align*}
{\Pr}_r(S_1, R_1)
=&\frac {|Z(S_1, R_1)|}{|S_1||R_1|}\underset{r \in [x_1, R_1]}{\underset{x_1 + Z(S_1, R_1)\in\frac {S_1}{Z(S_1, R_1)}}{\sum}|C_{R_1}(x_1)|}
\end{align*}
noting that $r \in [x_1, R_1]$ if and only if $r \in [x_1 + z, R_1]$ and $C_{R_1}(x_1) = C_{R_1}(x_1 + z)$ for all $z \in Z(S_1, R_1)$.
Now, by  Lemma \ref{lemma1}, we have
\begin{equation}\label{eqiso-1}
{\Pr}_r(S_1, R_1) = \frac {|Z(S_1, R_1)|}{|S_1|}\underset{r \in [x_1, R_1]}{\underset{x_1 + Z(S_1, R_1)\in\frac {S_1}{Z(S_1, R_1)}}\sum} \frac {1}{|[x_1,R_1]|}.
\end{equation}
Similarly, it can be seen that
\begin{equation}\label{eqiso-2}
{\Pr}_{\beta (r)}(S_2, R_2) = \frac {|Z(S_2, R_2)|}{|S_2|}\underset{\beta (r)\in [x_2, R_2]}{\underset{x_2 + Z(S_2, R_2)\in\frac {S_2}{Z(S_2, R_2)}}\sum}
\frac {1}{|[x_2, R_2]|}.
\end{equation}
Since $(\alpha,\beta)$ is a    $\mathbb Z$-isoclinism between $(S_1, R_1)$ and $(S_2, R_2)$ we have  $\frac {|S_1|}{|Z(S_1, R_1)|} = \frac {|S_2|}{|Z(S_2, R_2)|}$, $|[x_1,R_1]| = |[x_2, R_2]|$ and $r \in [x_1, R_1]$ if and only if  $\beta (r) \in [x_2, R_2]$. Hence, the result follows from \eqref{eqiso-1} and \eqref{eqiso-2}.
\end{proof}

\end{document}